\documentclass[11pt,reqno]{amsart}
\usepackage[all,cmtip]{xy}
\usepackage{color,graphics,epsfig,amsmath,a4wide,psfrag}
\usepackage{amsmath}
\usepackage{amsfonts}
\usepackage{amssymb}
\usepackage{graphicx}

\providecommand{\U}[1]{\protect\rule{.1in}{.1in}}

\numberwithin{equation}{section}

\newcommand{\mA}{\mathbb{A}}

\newcommand{\mC}{\mathbb{C}}
\newcommand{\mD}{\mathbb{D}}

\newcommand{\mT}{\mathbb{T}}

\newcommand{\mZ}{\mathbb{Z}}

\newcommand{\inv}{{\textrm{inv }}}

\newtheorem{theorem}{Theorem}[section]
\newtheorem{lemma}[theorem]{Lemma}

\theoremstyle{definition}

\theoremstyle{definition}
\newtheorem{definition}[theorem]{Definition}
\theoremstyle{definition}

\begin{document}
\title[$\nu$-metric for $H^{\infty}$]{Reformulation of the extension of the $\nu$-metric for $H^{\infty}$}
\author{Marie Frentz and Amol Sasane}
\address{Department of Mathematics, London School of Economics, Houghton Street, London
WC2A 2AE, United Kingdom}
\email{e.m.frentz@lse.ac.uk, sasane@lse.ac.uk}

\begin{abstract}
The classical $\nu$-metric introduced by Vinnicombe in robust control theory
for rational plants was extended to classes of nonrational transfer functions
in \cite{BS12}. In \cite{Sas12}, an extension of the classical $\nu$-metric
was given when the underlying ring of stable transfer functions is the Hardy
algebra, $H^{\infty}$. However, this particular extension to $H^{\infty}$ did
not directly fit in the abstract framework given in \cite{BS12}. In this paper
we show that the case of $H^{\infty}$ also fits into the general abstract
framework in \cite{BS12} and that the $\nu$-metric defined in this setting is
identical to the extension of the $\nu$-metric defined in \cite{Sas12}. This
is done by introducing a particular Banach algebra, which is the inductive
limit of certain $C^{\ast}$-algebras.

\end{abstract}
\maketitle

\section{Introduction}

The present paper deals with a fundamental problem in robust stabilization 
of linear control systems governed by PDEs/delay-differential equations. We 
 refer the uninitiated reader to the textbooks \cite{CZ}, \cite{ParB} (for an introduction to control theory 
in the PDE/delay-differential equation context using operator theoretic methods) and to the monograph  \cite{VinB} (for 
an introduction to robust control using frequency domain methods). 

We recall the general stabilization problem in
control theory. Suppose that $R$ is a commutative integral domain with
identity (thought of as the class of stable transfer functions) and let
$\mathbb{F}(R)$ denote the field of fractions of $R$ (thought of as the set of
unstable plants). The stabilization problem is then the following: given an
unstable plant transfer function $P\in\left(  \mathbb{F}(R)\right)  ^{p\times
m}$, find a stabilizing controller transfer function $C\in\left(
\mathbb{F}(R)\right)  ^{m\times p}$ such that
\[
H(P,C):=\left[
\begin{array}
[c]{l}%
P\\
I
\end{array}
\right]  \left(  I-CP\right)  ^{-1}\left[
\begin{array}
[c]{ll}%
-C & I
\end{array}
\right]  \in R^{(p+m)\times(p+m)}.
\]
Robust stabilization goes one step further; in many practical situations one
knows that the plant is merely an approximation of reality and therefore\ one
wishes that the controller $C$ not only stabilizes the nominal plant $P,$ but
also all plants $\widetilde{P}$, sufficiently close to $P$. A metric which
emerged from the need to define closeness of plants, is the so-called $\nu
$-metric, introduced by Vinnicombe in \cite{Vin93}, where it was shown that
stability is a robust property of the plant with respect to the $\nu$-metric.
However, $R$ was essentially taken to be the set of rational functions without
poles in the closed unit disk.

In \cite{BS12} the $\nu$-metric of Vinnicombe was extended in an abstract
manner, in order to cover the case when $R$ is a ring of stable transfer
functions of possibly infinite-dimensional systems. In particular, the set-up
for defining the abstract $\nu$-metric was as follows:

\begin{itemize}
\item[(A1)] $R$ is a commutative integral domain with identity.

\item[(A2)] $S$ is a unital commutative semisimple complex Banach algebra with
an involution $\cdot^{\ast}$, such that $R\subset S$.

\item[(A3)] With ${\text{inv }} S$ denoting the invertible elements of $S$,
there exists a map $\iota:{\text{inv }} S\rightarrow G$, where $(G,\star)$ is
an Abelian group with identity denoted by $\circ$, and $\iota$ satisfies:

\begin{itemize}
\item[(I1)] $\iota(ab)=\iota(a)\star\iota(b)$ for all $a,b\in{\text{inv }} S$,

\item[(I2)] $\iota(a^{\ast})=-\iota(a)$ for all $a\in{\text{inv }} S$,

\item[(I3)] $\iota$ is locally constant, that is, $\iota$ is continuous when
$G$ is equipped with the discrete topology.
\end{itemize}

\item[(A4)] $x\in R \bigcap{\text{inv }} S$ is invertible as an element of $R$
if and only if $\iota(x)=\circ$.
\end{itemize}

In \cite{BS12}, it was shown that the abstract $\nu$-metric defined in the
above framework (which is recalled in Definition~\ref{Def v-metric} below), is
a metric on the class of all stabilizable plants, and moreover, that
stabilizability is a robust property of the plant.

In \cite{Sas12}, an extension of the $\nu$-metric was given when $R=H^{\infty}
$, the Hardy algebra of bounded and holomorphic functions in the unit disk in
$\mathbb{C}$. However, the $\nu$-metric for $H^{\infty}$ which was defined
there, did not fit in the abstract framework of \cite{BS12} in a direct
manner. Indeed, the metric was defined with respect to a parameter $\rho$
(essentially by using the abstract framework specialized to the disk algebra
and looking at an annulus of radii $\rho$ and $1$), and then the limit as
$\rho\nearrow1$ was taken to arrive at a definition of an extended $\nu
$-metric.  (This is recalled in Definition~\ref{metric_from_Sas12} below.)

It is a natural question to ask if the extension of the $\nu$-metric for
$H^{\infty}$ given in \cite{Sas12} can be viewed as a special case of the
abstract framework in \cite{BS12} with an appropriate choice of the Banach
algebra $S$ and the index function $\iota$. In this paper we shall show that
this is indeed possible. Thus our result  gives further support to the abstract
framework developed in \cite{BS12}, and progress in the abstract framework of
\cite{BS12} would then also be applicable in particular to our specialization
when $R=H^{\infty}$. We will construct a unital commutative semisimple Banach
algebra $S$ and an associated index function $\iota:=W$ for which (A1)-(A4)
hold. Moreover, we prove that the resulting $\nu$-metric obtained as a result
of this specialization of the abstract $\nu$-metric defined in \cite{BS12} is
identical to the extension of the $\nu$-metric defined for $H^{\infty}$
previously in \cite{Sas12}.

The outline of the paper is as follows:

\begin{enumerate}
\item In Section~2 we introduce some notation.

\item In Section~3, when $R=H^{\infty}$, we construct a certain Banach
algebra, $S:=\underrightarrow{\lim}\;C_{b}(\mathbb{A}_{r})$, and an associated
index function $\iota:=W$ satisfying the assumptions (A1)-(A4).

\item In Section~4, we define the $\nu$-metric for $H^{\infty}$ obtained by
specializing the abstract $\nu$-metric of \cite{BS12} with these choices of
$R:=H^{\infty}$, $S:=\underrightarrow{\lim}\;C_{b}(\mathbb{A}_{r})$ and
$\iota:=W$.

\item In Section~5, we prove that the $\nu$-metric obtained for $H^{\infty}$
in this setup coincides with the $\nu$-metric for $H^{\infty}$ given in
\cite{Sas12}.

\item In Section~6, as an illustration of the computability of the
proposed $\nu$-metric, we give an example where we calculate the $\nu$-metric
when there is uncertainty in the location of the zero of the (nonrational)
transfer function.

\item Finally, in Section~7, we give the rationale behind our choice of $S$ by first 
showing that in the Hardy algebra context ($R=H^\infty$), 
some natural guesses for the sought pair $(S, \iota)$ fail. We also explore the intrinsic nature 
of our choice of $S=\underrightarrow{\lim}\;C_{b}(\mathbb{A}_{r})$, by realizing it 
as $C(X)$ for an appropriate compact Hausdorff space $X$, and by showing its relation  with $L^\infty(\mT)$. 
\end{enumerate}

$\;$ 

\noindent {\bf Acknowledgements:} The second author gratefully acknowledges 
several useful discussions with Professors Ronald Douglas and Raymond Mortini  
pertaining to Section~7. 

\section{Notation\label{SecNotation}}

In this section we will fix some notation which will be used throughout the article.

Let $\cdot^{\ast}$ denote the involution in the Banach algebra, mentioned in
(A2). For $F\in S^{p\times m}$, the notation $F^{\ast}\in S^{m\times p}$
denotes the matrix given by $(F^{\ast})_{ij}=\left(  F_{ji}\right)  ^{\ast}$
for $1\leq i\leq p$ and $1\leq j\leq m.$ Here $(\cdot)_{ij}$ is used to
denote  the entry in the $i$th row and $j$th column of a matrix.

Let $\mathbb{F}(R)$ denote the field of fractions of $R.$ Given a matrix
$P\in(\mathbb{F}(R))^{p\times m}$, a factorization $P=ND^{-1},$ where $N$ and
$D$ are matrices with entries from $R,$ is called a \emph{right coprime
factorization of} $P$ if there exist matrices $X, Y$ with entries from $R$,
such that $XN+YD=I_{m}$. If, in addition, $N^{\ast}N+D^{\ast}D=I_{m}$, then
the right coprime factorization is referred to as a \emph{normalized right
coprime factorization of} $P$.

Given a matrix $P\in(\mathbb{F}(R))^{p\times m}$, a factorization
$P=\widetilde{D}^{-1}\widetilde{N},$ where $\widetilde{D}$ and $\widetilde{N}$
are matrices with entries from $R,$ is called a \emph{left coprime
factorization of} $P$ if there exist matrices $\widetilde{X},\ \widetilde{Y}$
with entries from $R,$ such that $\widetilde{N}\widetilde{X}+\widetilde{D}
\widetilde{Y}=I_{p}.$ If, in addition, $\widetilde{N}\widetilde{N}^{\ast
}+\widetilde{D} \widetilde{D}^{\ast}=I_{p},$ then the left coprime
factorization is referred to as a \emph{normalized left coprime factorization
of} $P$.

Let $\mathbb{S}(R,p,m)$ denote the set of all elements $P\in(\mathbb{F}%
(R))^{p\times m}$ that possess normalized right- and left coprime
factorizations. For $P\in\mathbb{S}(R,p,m),$ with factorizations
$P=\widetilde{D}^{-1}\widetilde{N}=ND^{-1},$ $G$ and $\widetilde{G}$ are
defined by
\begin{equation}
G:=%
\begin{bmatrix}
N\\
D
\end{bmatrix}
\quad\text{ and }\quad\widetilde{G}:=%
\begin{bmatrix}
-\widetilde{D} & \widetilde{N}%
\end{bmatrix}
.\label{Gena}%
\end{equation}
Further, we will define a norm on matrices with entries in $S$ using the
Gelfand transform.

\begin{definition}
Let $\mathfrak{M}(S)$ denote the maximal ideal space of the Banach algebra
$S$. For a matrix $M\in S^{p\times m},$ we define
\begin{equation}
\|M\|_{S,\infty}=\max_{\varphi\in\mathfrak{M}(S)}\,\rule[-.6ex]{.13em}{2.3ex}
\, \mathbf{M}(\varphi)\,\rule[-.6ex]{.13em}{2.3ex}\,,\label{max}%
\end{equation}
where $\mathbf{M}$ denotes the entry-wise Gelfand transform of $M$, and
$\,\rule[-.6ex]{.13em}{2.3ex}\, \cdot\,\rule[-.6ex]{.13em}{2.3ex}\,$ denotes
the induced operator norm from $\mathbb{C}^{m}$ to $\mathbb{C}^{p}$. (For the
sake of concreteness, we assume that $\mathbb{C}^{m}$ and $\mathbb{C}^{p}$ are
both equipped with the usual Euclidean $2$-norm.)
\end{definition}

The maximum in \eqref{max} exists since $\mathfrak{M}(S)$ is a compact space
when equipped with the Gelfand topology, that is, the weak-$\ast$ topology
induced from $\mathcal{L}(S;\mathbb{C}),$ the set of continuous linear
functionals from $S$ to $\mathbb{C}$. Moreover, since $S$ is semisimple, the
Gelfand transform,
\[
\widehat{\cdot}:S\rightarrow\widehat{S}\subset C(\mathfrak{M}(S),\mathbb{C}),
\]
is an injective algebra homomorphism by the Gelfand-Naimark theorem.

\section{Validity of (A1)-(A4) with $R=H^{\infty}$, $S=\displaystyle \lim
_{\longrightarrow}C_{b}(\mathbb{A}_{r})$ and $\iota=W$}

In this section we construct a Banach algebra $S$ and an index function
$\iota$ such that the assumptions (A1)-(A4) are satisfied for $R=H^{\infty}$.

In order to construct $S$, we will use the notion of inductive limits of
$C^{\ast}$-algebras. We refer the reader to \cite[Section 2.6]{LOT99} and
\cite[Appendix L]{Weg93} for background on the inductive limit of $C^{\ast}$-algebras.

The \emph{Hardy algebra} $H^{\infty}$ consists of all bounded and holomorphic
functions defined on the open unit disk $\mathbb{D}:=\{z\in\mathbb{C}
:|z|<1\}$, with pointwise operations and the usual supremum norm
\[
\|f\|_{\infty}= \displaystyle \sup_{z\in\mathbb{D}}|f(z)|, \quad f\in
H^{\infty}.
\]
For given $r\in(0,1),$ let
\[
\mathbb{A}_{r}:=\{z\in\mathbb{C}:r<|z|<1\}
\]
denote the open annulus and let $C_{b}(\mathbb{A}_{r})$ be the $C^{\ast}
$-algebra of all bounded and continuous functions $f:\mathbb{A}_{r}
\rightarrow\mathbb{C}$, equipped with pointwise operations and the supremum
norm: for $f\in C_{b}(\mathbb{A}_{r})$ we define
\[
\|f\|_{L^{\infty}(\mathbb{A}_{r})}:=\sup_{z\in\mathbb{A}_{r}}|f(z)|.
\]
When $\mathbb{A}_{r}$ is implicitly understood we will write $\|\cdot
\|_{L^{\infty}}$ instead of $\|\cdot\|_{L^{\infty}(\mathbb{A}_{r})}.$
Moreover, for $0<r\leq R<1$ we define the map $\pi_{r}^{R}:C_{b}
(\mathbb{A}_{r})\rightarrow C_{b}(\mathbb{A}_{R})$ by restriction:
\[
\pi_{r}^{R}(f)=f|_{\mathbb{A}_{R}}, \quad f\in C_{b} (\mathbb{A}_{r}).
\]
Consider the family $\left(  C_{b}(\mathbb{A}_{r}),\pi_{r}^{R}\right)  $ for
$0<r\leq R<1$. We note that

\begin{itemize}
\item[(i)] $\pi_{r}^{r}$ is the identity map on $C_{b}(\mathbb{A}_{r})$, and

\item[(ii)] $\pi_{r}^{R}\circ\pi_{\rho}^{r}=\pi_{\rho}^{R}$ for all
$0<\rho\leq r\leq R<1.$
\end{itemize}

Now consider the $\ast$-algebra
\[
{\displaystyle\prod\limits_{r\in(0,1)}} C_{b}(\mathbb{A}_{r}),
\]
and denote by $\mathcal{A}$ its $\ast$-subalgebra consisting of all elements
$f=(f_{r})=(f_{r})_{r\in(0,1)}$ such that there is an index $r_{0}$ with
$\pi_{r}^{R}(f_{r})=f_{R}$ for all $0<r_{0}\leq r\leq R<1$. Since every
$\pi_{r}^{R}$ is norm decreasing, the net $(\|f_{r}\|_{L^{\infty}
(\mathbb{A}_{r})})$ is convergent and we define
\[
\|f\|:=\lim_{r\rightarrow1}\|f_{r}\|_{L^{\infty}(\mathbb{A}_{r})}.
\]
Clearly this defines a seminorm on $\mathcal{A}$ that satisfies the $C^{\ast}
$-norm identity, that is,
\[
\|f^{\ast}f\|=\|f\|^{2},
\]
where $\cdot^{\ast}$ is the involution, that is, complex conjugation, see
(\ref{involution}) below. Now, if $N$ is the kernel of $\|\cdot\|$, then the
quotient $\mathcal{A}/N$ is a $C^{\ast}$-algebra (and we denote the norm again
by $\|\cdot\|$). This algebra is the so-called \emph{direct} or
\emph{inductive limit} of $(C_{b}(\mathbb{A}_{r}),\pi_{r}^{R})$ and we denote
it by
\[
\underrightarrow{\lim}\;C_{b}(\mathbb{A}_{r}).
\]
To every element $f\in C_{b}(\mathbb{A}_{r_{0}}),$ we associate a sequence
$f_{1}=(f_{r})$ in $\mathcal{A}$, where
\begin{equation}
f_{r}=\left\{
\begin{array}
[c]{ll}%
0 & \text{if }0<r<r_{0},\\
\pi_{r_{0}}^{r}(f) & \text{if }r_{0}\leq r<1.
\end{array}
\right. \label{f1}%
\end{equation}
We also define a map $\pi_{r}:C_{b}(\mathbb{A}_{r})\rightarrow
\underrightarrow{\lim}\;C_{b}(\mathbb{A}_{r})$ by
\[
\pi_{r}(f):=[f_{1}],\quad f\in C_{b}(\mathbb{A}_{r}),
\]
where $[f_{1}]$ denotes the equivalence class in $\underrightarrow{\lim
}\;C_{b}(\mathbb{A}_{r})$ which contains $f_{1}$. We will use the fact that
the maps $\pi_{r}$ are in fact $\ast$-homomorphisms. We note that these maps
are compatible with the connecting maps $\pi_{r}^{R}$ in the sense that every
diagram shown below is commutative.
\[
\xymatrix{
C_b(\mA_r) \ar[r]^{\pi_r^R} \ar[rd]_{\pi_r} & C_b(\mA_R) \ar[d]^{\pi_R} \\ &
\displaystyle {\lim_{\longrightarrow} C_b(\mA_r)}
}
\]

\subsection{Verification of assumption (A2)}

We note that $\underrightarrow{\lim}\;C_{b}(\mathbb{A}_{r})$ is a complex
commutative Banach algebra with involution, see for instance \cite[Section
2.6]{LOT99}. The multiplicative identity arises from the constant function
$f\equiv1 $ in $C_{b}(\mathbb{A}_{0})$, that is, $\pi_{0}(f)$. Moreover, we
can define an involution in $C_{b}(\mathbb{A}_{r})$ by setting
\begin{equation}
(f^{\ast})(z):=\overline{f(z)},\ \ \ z\in\mathbb{A}_{r},\label{involution}%
\end{equation}
and this implicitly defines an involution of elements in
$\underrightarrow{\lim}\; C_{b}(\mathbb{A}_{r})$.

It remains to prove that $\underrightarrow{\lim}\;C_{b}(\mathbb{A}_{r})$ is
semisimple and that $H^{\infty}\subset\underrightarrow{\lim}\;C_{b}%
(\mathbb{A}_{r})$. That $\underrightarrow{\lim}\;C_{b}(\mathbb{A}_{r})$ is
semisimple follows immediately since all commutative $C^{\ast}$-algebras are
semisimple. To see that this is the case, recall that a Banach algebra is
\emph{semisimple} if its radical ideal (that is, the intersection of all its
maximal ideals) is zero. Since $\underrightarrow{\lim}\;C_{b}(\mathbb{A}_{r})$
is a $C^{\ast}$-algebra, the Gelfand-Naimark theorem asserts that the Gelfand
transform is an isometric isomorphism of $\underrightarrow{\lim}%
\;C_{b}(\mathbb{A}_{r})$ onto $C(\Delta)$, where $\Delta$ is maximal ideal
space of $\underrightarrow{\lim}\;C_{b}(\mathbb{A}_{r}).$ Since the maximal
ideal space of $C(\Delta)$ comprises just point evaluations at $x\in\Delta$,
$C(\Delta)$ is semisimple, and it follows that $\underrightarrow{\lim}%
\;C_{b}(\mathbb{A}_{r})$ is semisimple too.

Finally, note that there is a natural embedding of $H^{\infty}$ into
$\underrightarrow{\lim}\; C_{b}(\mathbb{A}_{r})$, namely
\begin{equation}
f\mapsto\pi_{0}(f):H^{\infty}\longrightarrow\underrightarrow{\lim}\; C_{b}
(\mathbb{A}_{r}).\label{injektiv}%
\end{equation}
This is an injective map since $\pi_{0}$ is linear and if $\pi_{0}(f)=[(0)] $
for $f\in H^{\infty}$, then
\[
\lim_{r\rightarrow1}\left(  \sup_{z\in\mathbb{A}_{r}}|f(z)|\right)  =0,
\]
and so, in particular, the radial limit
\[
\lim_{r\rightarrow1}f(re^{i\theta})=0
\]
for all $\theta\in\lbrack0,2\pi)$. By the uniqueness of the boundary function
for $H^{\infty}$ functions (see for example \cite[Theorem~17.18]{Rud}), this
implies that $f=0$ in $H^{\infty}$. Thus $H^{\infty}$ can be considered to be
a subset of $\underrightarrow{\lim}\; C_{b}(\mathbb{A}_{r})$ (via the
injective restriction of map $\pi_{0}$ to $H^{\infty}$).

\subsection{Verification of assumption (A3)}

We now construct an index function
\[
\iota: {\text{inv }} \left(  \underrightarrow{\lim}\; C_{b}(\mathbb{A}
_{r})\right)  \rightarrow G,
\]
with a certain choice of an Abelian group $(G,\star)$ satisfying (I1)-(I3). We
will take the Abelian group $(G,\star)$ to be the additive group
$(\mathbb{Z},+)$ of integers and we will define $\iota$ in terms of winding numbers.

Let $C(\mathbb{T})$ denote the Banach algebra of complex valued continuous
functions on the unit circle $\mathbb{T}:=\{z\in\mathbb{C}:|z|=1\}$. If
$f\in{\text{inv }} C(\mathbb{T})$, we denote by $w(f)\in\mathbb{Z}$ its
winding number, see for instance \cite[p. 57]{Ull08}. For $f\in{\text{inv }}
(C_{b}(\mathbb{A}_{\rho}))$ and for $0<\rho<r<1$ we define the map
$f^{r}:\mathbb{T}\rightarrow\mathbb{C}$ by
\[
f^{r}(\zeta)=f(r\zeta),\quad\zeta\in\mathbb{T}.
\]
If $f\in{\text{inv }} (C_{b}(\mathbb{A}_{\rho}))$, then $f^{r}\in{\text{inv }}
(C(\mathbb{T}))$, and this implies that $f^{r}$ has a well defined integral
winding number $w(f^{r})\in\mathbb{Z}$ with respect to $0$. In
\cite[Proposition 3]{Sas12} it is proved that for $f\in\ $inv$(C_{b}
(\mathbb{A}_{\rho}))$ and $0<\rho<r<r^{\prime}<1,$
\begin{equation}
w(f^{r})=w(f^{r^{\prime}}),\label{win}%
\end{equation}
by the local constancy of the winding number.

Let $[(f_{r})]\in{\text{inv }} (\underrightarrow{\lim}\ C_{b}(\mathbb{A}_{r}
))$. Then there exists $[(g_{r})]\in{\text{inv }} (\underrightarrow{\lim
}\;C_{b}(\mathbb{A}_{r}))$ such that
\[
\lbrack(f_{r})][(g_{r})]=[(1)].
\]
Thus there exist $(f_{r}^{1})\in\lbrack(f_{r})]$, $\left(  g_{r}^{1}\right)
\in\lbrack(g_{r})]$ and $r_{0}$ such that for $r\in(r_{0},1),$
\[
f_{r}^{1}g_{r}^{1}=1
\]
pointwise. In particular, the image $f_{r}^{1}(\mathbb{A}_{r}),$ of $f_{r}
^{1}$ is a set in $\mathbb{C}$ that is bounded away from zero, that is, there
exists a $\delta>0$ such that $f_{r}^{1}(\mathbb{A}_{r})\bigcap\{z\in
\mathbb{C}:|z|<\delta\}=\emptyset.$ If $(f_{r}^{2})\in\lbrack(f_{r})]$ is
another sequence, then there exists $\widetilde{r}_{0}\in(0,1)$ such that
$\sup_{z\in\mathbb{A}_{r}}|f_{r}^{1}(z)-f_{r}^{2}(z)|<\delta/2$ for all
$r\in(\widetilde{r}_{0},1).$ Therefore, if we let $r$ be such that
$\max\{r_{0},\widetilde{r}_{0}\}<r<1,$ we can look at the restrictions
\[
f_{r}^{1}|_{\mathbb{T}_{\rho}},\ f_{r}^{2}|_{\mathbb{T}_{\rho}},
\]
and these will have the same winding number, since the graph of $f_{r}
^{1}|_{\mathbb{T}_{\rho}}$is at least at a distance $\delta$ from the origin,
while the distance from $f_{r}^{2}|_{\mathbb{T}_{\rho}}$ to $f_{r}
^{1}|_{\mathbb{T}_{\rho}}$ is smaller than $\delta/2.$ Thus $f_{r}
^{2}|_{\mathbb{T}_{\rho}}$ must wind around the origin the same number of
times as $f_{r}^{1}|_{\mathbb{T}_{\rho}}$ and their winding numbers coincide,
see \cite[Proposition 4.12]{Ull08}. With this in mind, we define the map
$W:{\text{inv }} (\underrightarrow{\lim}\;C_{b}(\mathbb{A}_{r}))\rightarrow
\mathbb{Z}$ by
\begin{equation}
W(f)=\lim_{r\rightarrow1}w(f_{r}|_{\mathbb{T}_{\rho}}),\text{ for }
f=[(f_{r})]\in{\text{inv }} (\underrightarrow{\lim}\;C_{b}(\mathbb{A}
_{r})),\;\rho\in(r,1).\label{defi}%
\end{equation}
Since two sequences in the same equivalence class will have the same winding
number eventually as $r\nearrow1$, it is enough to consider only one of the
sequences in the equivalence class in \eqref{defi}. We take $\iota=W$. By
\eqref{win}, \eqref{defi} and the definition of winding numbers, it follows
that (I1) and (I2) hold. Finally, analogous to the proof of
\cite[Proposition~6]{Sas12}, it can be verified that assumption (I3) also holds.

\subsection{Verification of assumption (A4)}

In light of \eqref{injektiv}, we can view $H^{\infty}$ as a subset of
$\underrightarrow{\lim}\; C_{b}(\mathbb{A}_{r})$.

Let $f\in H^{\infty}\bigcap{\text{inv }} (\underrightarrow{\lim}\;
C_{b}(\mathbb{A}_{r}))$. First, assume that $f$ is invertible in $H^{\infty}$
and let $g\in H^{\infty}$ be its inverse. For each $r\in(0,1)$ we can define
$f^{r}(z):=f(rz)\in A(\mathbb{D}),$ and since $f$ is invertible in $H^{\infty
},$ $f^{r}$ is invertible in $A(\mathbb{D})$ and on $C(\mathbb{T)}$. By the
Nyquist criterion \cite[Lemma 5.2]{BS12} for $A(\mathbb{D})$ this implies that
$w(f^{r})=0$. Because of the homotopic invariance of winding numbers,
$w(f^{r})=w(f|_{\mathbb{T}_{r}}),$ and this implies that
\[
W(f)=\lim_{r\rightarrow1}w(f|_{\mathbb{T}_{r}})=0.
\]

Next, assume that $f\in H^{\infty}\bigcap{\text{inv }} (\underrightarrow{\lim
}\; C_{b}(\mathbb{A}_{r}))$ and that $W(f)=0$. Let $F=\pi_{0}(f)$ and let
$G\in\underrightarrow{\lim}\ C_{b}(\mathbb{A}_{r})$ be the inverse of $F$.
Again, for $r\in(0,1)$ we define $f^{r}(z):=f(rz)\in A(\mathbb{D}).$ Since
$\pi_{0}(f)\in\ $inv$(\underrightarrow{\lim}\ C_{b}(\mathbb{A}_{r})),$ we have
that $f^{r}\in\ $inv$(C(\mathbb{T}))$. We know that
\[
W(f)=\lim_{r\rightarrow1}w(f_{r}|_{\mathbb{T}_{\rho}})=0.
\]
Using the fact that the winding number is integer valued, and using the local
constancy of winding numbers, it follows that $w(f^{r}|_{\mathbb{T} })=0$ for
$r$ close enough to $1.$ Moreover, the Nyquist criterion referred to above
implies that $f^{r}$ is invertible in $A(\mathbb{D})$. In particular, this
means that $f(rz)\neq0$ for all $z\in\mathbb{D}$. Since this is the case for
all $r$ large enough, $f(z)\neq0$ for all $z\in\mathbb{D}$. This implies that
$f$ has a pointwise inverse, say $g$, and this $g$ is holomorphic. What
remains to be proved is that $g$ is bounded. To this end, we consider
$(f_{r})\in F$ as defined in \eqref{f1}, and its inverse $(g_{r})\in G$, and
we note that there exists $\rho\in(0,1)$ such that for all $r\in(\rho,1),$
\[
f_{r}(z)g_{r}(z)=f(z)g(z)=1,\quad z\in\mathbb{A}_{r}.
\]
The maximum modulus principle then gives us that
\[
\sup_{\mathbb{D}}|g(z)|=\sup_{\mathbb{A}_{r}}|g(z)|\leq\|g_{r}\|_{L^{\infty
}(\mathbb{A}_{r})}<+\infty.
\]
That is, $g$ is bounded and hence $g\in H^{\infty}.$

\medskip

Summarizing, in this section we have checked that with
\[
R:=H^{\infty},\quad S:=\underrightarrow{\lim}\;C_{b}(\mathbb{A}_{r}
),\quad\text{and}\quad\iota:=W,
\]
the assumptions (A1)-(A4) from \cite{BS12} (which we recalled at the outset)
are all satisfied, and so the abstract $\nu$-metric given in \cite{BS12} is
applicable when the ring of stable transfer functions is the Hardy algebra
$H^{\infty}$. In the next section, we will clarify the explicit form taken by
abstract $\nu$-metric in this specialization when $(R,S,\iota)=(H^{\infty
},\underrightarrow{\lim}\;C_{b}(\mathbb{A}_{r}),W)$.

\section{The abstract $\nu$-metric when $(R,S,\iota)=(H^{\infty}
,\displaystyle\lim_{\longrightarrow}C_{b}(\mathbb{A}_{r}),W)$}

We will now present the abstract $\nu$-metric from \cite{BS12}, applied to our
special case
\[
(R,S,\iota)=(H^{\infty},\underrightarrow{\lim}\;C_{b}(\mathbb{A}_{r}),W).
\]
Before doing so, we mention that we will use notation below which is 
analogous to the one given earlier in \eqref{Gena}: thus $G_i, {\widetilde{G}}_i$ $i=1,2$ below 
are understood to be 
$$
G_i:=
\begin{bmatrix}
N_i\\
D_i
\end{bmatrix}
\quad\text{ and }\quad{\widetilde{G}}_i:=
\begin{bmatrix}
-{\widetilde{D}}_i & {\widetilde{N}}_i
\end{bmatrix}, \quad i=1,2.
$$ 
The $\nu$-metric for stabilizable plants over $H^{\infty}$ is
then defined as follows.

\begin{definition}
\label{Def v-metric} For $P_{1},\ P_{2}\in\mathbb{S}(H^{\infty},p,m),$ with
normalized left/right coprime factorizations
\begin{align*}
P_{1}  &  =N_{1}D_{1}^{-1}=\widetilde{D}_{1}^{-1}\widetilde{N}_{1},\\
P_{2}  &  =N_{2}D_{2}^{-1}=\widetilde{D}_{2}^{-1}\widetilde{N}_{2},
\end{align*}
the $\nu$-metric $d_{\nu}$ is given by
\[
d_{\nu}(P_{1},P_{2})\!=\!\left\{
\begin{array}
[c]{ll}%
\!\!\!\Vert\widetilde{G}_{2}G_{1}\Vert_{\underrightarrow{\lim}\ C_{b}
(\mathbb{A}_{r}),\infty} & \text{if }\det(G_{1}^{\ast}G_{2})\in{\text{inv }
}(\underrightarrow{\lim}\;C_{b}(\mathbb{A}_{r}))\text{ and }W(\det(G_{1}
^{\ast}G_{2}))=0,\\
\!\!1 & \text{otherwise.}%
\end{array}
\right.
\]
\end{definition}

Although the normal coprime factorization is not unique for a given plant,
$d_{\nu} $ is still a well-defined metric on $\mathbb{S}(R,p,m)$; see
\cite[Theorem 3.1]{BS12}.

The next step is to show that, for $\widetilde{G}_{2}G_{1}\in(H^{\infty
})^{p\times m}$, the norm
\[
\Vert\widetilde{G}_{2}G_{1}\Vert_{\underrightarrow{\lim}\ C_{b}(\mathbb{A}
_{r}),\infty}
\]
above can be replaced by the usual $H^{\infty}$-norm $\Vert\widetilde{G}
_{2}G_{1}\Vert_{\infty}$ . This will simplify the calculation of the $\nu
$-metric, and will also help us to show that the $\nu$-metric defined above
and the extension of the $\nu$-metric for $H^{\infty}$ given in in
\cite{Sas12} are the same. We show the following, analogous to  the result in
\cite[Lemma~3.12]{Sas12}.

\begin{theorem}
\label{tumm} For $F\in(H^{\infty})^{p\times m},$ there holds
\[
\Vert\lbrack(F_{r})]\Vert_{\underrightarrow{\lim}\;C_{b}(\mathbb{A}
_{r}),\infty}=\Vert F\Vert_{\infty}:=\displaystyle\sup_{z\in\mathbb{D}%
}|F(z)|.
\]

\end{theorem}

\begin{proof}
Suppose first that $p=m=1$. Then
\begin{align*}
\|F\|_{\infty}  &  =\sup_{z\in\mathbb{D}}|F(z)| =\lim_{r\rightarrow1}
\sup_{z\in\mathbb{A}_{r}}|F(z)| =\|\pi_{0}(F)\|_{\underrightarrow{\lim}
\ C_{b}(\mathbb{A}_{r})}\\
&  =\max_{\varphi\in\mathfrak{M}(\underrightarrow{\lim}\; C_{b}(\mathbb{A}
_{r}))} \|\widehat{\pi_{0}(F)}(\varphi)\| =\|\pi_{0}
(F)\|_{\underrightarrow{\lim}\ C_{b}(\mathbb{A}_{r}),\infty.}%
\end{align*}
The last equality follows from the Gelfand-Naimark Theorem; see
\cite[Theorem~11.18]{Rud91}. This proves the theorem for $p=m=1.$

Let us assume that at least one of $p$ and $m$ are larger than $1$. To treat
this case, we introduce the notation $\sigma_{\max}(X)$ for $X\in
\mathbb{C}^{p\times m}$, denoting the largest singular value of $X,$ that is,
the square root of the largest eigenvalue of $XX^{\ast}$ (or $X^{\ast}X$). In
particular, we note that the map $\sigma_{\max}(\cdot):\mathbb{C}^{p\times
m}\rightarrow\lbrack0,\infty)$ is continuous. Let
\[
F=[(F_{r})]\in\left(  \underrightarrow{\lim}\;C_{b}(\mathbb{A}_{r})\right)
^{p\times m}.
\]
Then $\sigma_{\max}(\widehat{F}(\cdot))$ is a continuous function on the
maximal ideal space $\mathfrak{M}(\underrightarrow{\lim}\;C_{b}(\mathbb{A}%
_{r})),$ and so again by the Gelfand-Naimark Theorem, there exists an element
$\mu_{1}\in\underrightarrow{\lim}\ C_{b}(\mathbb{A}_{r})$ such that
\[
\widehat{\mu}_{1}(\varphi)=\sigma_{\max}(\widehat{F}(\varphi))\text{ \ \ for
all }\varphi\in\mathfrak{M}(\underrightarrow{\lim}\ C_{b}(\mathbb{A}_{r})).
\]
Define $\mu_{2}:=[(\sigma_{\max}(F_{r}(\cdot)))]\in\underrightarrow{\lim
}\;C_{b}(\mathbb{A}_{r}).$ For fixed $r$, we have that $\det((\mu_{2}%
)_{r})^{2}I-F_{r}^{\ast}F_{r})=0$ in $C_{b}(\mathbb{A}_{r})$, which implies
that $\det\left(  \mu_{2}^{2}I-F^{\ast}F\right)  =0$ in $\underrightarrow{\lim
}\;C_{b}(\mathbb{A}_{r})$. Taking Gelfand transforms, we obtain
\[
\det\left(  (\widehat{\mu}_{2}(\varphi))^{2}I-(\widehat{F}(\varphi))^{\ast
}(\widehat{F}(\varphi))\right)  =0
\]
and so $|\widehat{\mu}_{2}(\varphi)|\leq\sigma_{\max}\left(  \widehat{F}%
(\varphi)\right)  =\widehat{\mu}_{1}(\varphi)$ for all $\varphi\in
\mathfrak{M}(\underrightarrow{\lim}\;C_{b}(\mathbb{A}_{r}))$. Thus
\begin{equation}
\Vert\mu_{2}\Vert_{\underrightarrow{\lim}\ C_{b}(\mathbb{A}_{r})}\leq\Vert
\mu_{1}\Vert_{\underrightarrow{\lim}\ C_{b}(\mathbb{A}_{r})}.\label{rr1}%
\end{equation}
On the other hand, since
\[
\det\left(  (\widehat{\mu}_{1}(\varphi))^{2}I-(\widehat{F}(\varphi))^{\ast
}(\widehat{F}(\varphi))\right)  =0\text{ \ \ for all }\varphi\in
\mathfrak{M}(\underrightarrow{\lim}\;C_{b}(\mathbb{A}_{r})),
\]
it follows that $\det\left(  \mu_{1}^{2}I-F^{\ast}F\right)  =0$ in
$\underrightarrow{\lim}\ C_{b}(\mathbb{A}_{r})$. Hence, for all $\epsilon>0$,
there exists $r_{0}\in(0,1)$ such that for all $r>r_{0}$, $|(\mu_{1}%
)_{r}(z)|\leq\sigma_{\max}(F_{r}(z))+\epsilon=(\mu_{2})_{r}(z)+\epsilon$, for
$z\in\mathbb{A}_{r}$. So
\begin{equation}
\Vert\mu_{1}\Vert_{\underrightarrow{\lim}\;C_{b}(\mathbb{A}_{r})}\leq\Vert
\mu_{2}\Vert_{\underrightarrow{\lim}\ C_{b}(\mathbb{A}_{r})}+\epsilon
.\label{rr2}%
\end{equation}
As the choice of $\epsilon$ was arbitrary, (\ref{rr1}) and (\ref{rr2}) imply
that $\Vert\mu_{1}\Vert_{\underrightarrow{\lim}\;C_{b}(\mathbb{A}_{r})}%
=\Vert\mu_{2}\Vert_{\underrightarrow{\lim}\;C_{b}(\mathbb{A}_{r})}.$ Using
this observation, we have that
\begin{align*}
\Vert\lbrack(F_{r})]\Vert_{\underrightarrow{\lim}\;C_{b}(\mathbb{A}%
_{r}),\infty} &  =\max_{\varphi\in\mathfrak{M}(\underrightarrow{\lim}%
\;C_{b}(\mathbb{A}_{r}))}\widehat{\mu}_{1}(\varphi)\quad\text{(definition)}%
\phantom{\lim_{r\rightarrow1}}\\
&  =\Vert\mu_{1}\Vert_{\underrightarrow{\lim}\ C_{b}(\mathbb{A}_{r}%
)}\phantom{\lim_{r\rightarrow1}}\\
&  =\Vert\mu_{2}\Vert_{\underrightarrow{\lim}\ C_{b}(\mathbb{A}_{r}%
)}\phantom{\lim_{r\rightarrow1}}\\
&  =\lim_{r\rightarrow1}\Vert\sigma_{\max}(F_{r}(\cdot))\Vert_{L^{\infty
}(\mathbb{A}_{r})}\\
&  =\lim_{r\rightarrow1}\sup_{z\in\mathbb{A}_{r}}\sigma_{\max}(F_{r}(z))\\
&  =\lim_{r\rightarrow1}\sup_{z\in\mathbb{A}_{r}}\,\rule[-0.6ex]%
{0.13em}{2.3ex}\,F_{r}(z)\,\rule[-0.6ex]{0.13em}{2.3ex}\,\\
&  =\Vert F\Vert_{\infty}\quad\text{(since }F\in(H^{\infty})^{p\times
m}\text{)}\phantom{\lim_{r\rightarrow1}}
\end{align*}
This completes the proof.
\end{proof}

Hence abstract $\nu$-metric from \cite{BS12}, when applied to our special
case
\[
(R,S,\iota)=(H^{\infty},\underrightarrow{\lim}\;C_{b}(\mathbb{A}_{r}),W)
\]
now takes the following explicit form.

\begin{definition}
\label{Def v-metric 2} For $P_{1},\ P_{2}\in\mathbb{S}(H^{\infty},p,m),$ with
normalized left/right coprime factorizations
\begin{align}
P_{1}  &  =N_{1}D_{1}^{-1}=\widetilde{D}_{1}^{-1}\widetilde{N}_{1},\nonumber\\
P_{2}  &  =N_{2}D_{2}^{-1}=\widetilde{D}_{2}^{-1}\widetilde{N}_{2}
,\label{cop 2}%
\end{align}
the $\nu$-metric $d_{\nu}$ is given by
\begin{equation}
d_{\nu}(P_{1},P_{2})\!=\!\left\{
\begin{array}
[c]{ll}%
\!\!\!\Vert\widetilde{G}_{2}G_{1}\Vert_{\infty} & \text{if }\det(G_{1}^{\ast
}G_{2})\in{\text{inv }} (\underrightarrow{\lim}\;C_{b}(\mathbb{A}_{r}))\text{
and }W(\det(G_{1}^{\ast}G_{2}))=0,\\
\!\!1 & \text{otherwise.}%
\end{array}
\right. \label{abstract_2}%
\end{equation}

\end{definition}

\section{The $\nu$-metric for $H^{\infty}$ given by (\ref{abstract_2})
coincides with the one given in \cite{Sas12}}

The aim of this section is to prove that the extension of the $\nu$-metric
given in \cite{Sas12} coincides with the $\nu$-metric given by
\eqref{abstract_2}, which, as we have seen in the previous section, is a
specialization of the abstract $\nu$-metric defined in \cite{BS12} when
$(R,S,\iota)=(H^{\infty},\underrightarrow{\lim}\;C_{b}(\mathbb{A}_{r}),W)$.

Let us first recall the extension of the $\nu$-metric for $H^{\infty}$ given
in \cite{Sas12}. In order to distinguish it from the metric $d_{\nu}$ given by
\eqref{abstract_2}, we denote the metric from \cite{Sas12} by $\widetilde{d}
_{\nu}$.

\begin{definition}
\label{metric_from_Sas12} For $P_{1},\ P_{2}\in\mathbb{S}(H^{\infty},p,m),$
with normalized left/right coprime factorizations as in \eqref{cop 2}, let
\[
\widetilde{d}_{\nu}^{\rho}(P_{1},P_{2}):=\left\{
\begin{array}
[c]{ll}%
\Vert\widetilde{G}_{2}G_{1}\Vert_{\infty} & \text{if }\det(G_{1}^{\ast}
G_{2})\in{\text{inv }} C_{b}(\mathbb{A}_{\rho})\text{ and }w(\det(G_{1}^{\ast
}G_{2})|_{\mathbb{T}_{r}})=0,\ r\in(\rho,1),\\
1 & \text{otherwise.}%
\end{array}
\right.
\]
Then, the extended $\nu$-metric for $H^{\infty}$ is defined by
\begin{equation}
\widetilde{d}_{\nu}(P_{1},P_{2}):=\lim_{\rho\rightarrow1}\widetilde{d}
_{v}^{\rho}(P_{1},P_{2}).\label{extended}%
\end{equation}

\end{definition}

\begin{theorem}
On the set $\mathbb{S}(H^{\infty},p,m)$ of stabilizable plants, the metric
$d_{\nu}$ given by in \eqref{abstract_2}, and the metric $\widetilde{d}_{v}$
given by \eqref{extended}, coincide.
\end{theorem}

\begin{proof}
First, suppose that $d_{\nu}(P_{1},P_{2})<1$. Then we have $\det(G_{1}^{\ast
}G_{2} )\in{\text{inv }} (\underrightarrow{\lim}\;C_{b}(\mathbb{A}_{r}))$ and
$W(\det(G_{1}^{\ast}G_{2}))=0$. Since $\det(G_{1}^{\ast}G_{2})$, viewed as an
element of $\underrightarrow{\lim}\;C_{b}(\mathbb{A}_{r})$ via the map
$\pi_{0}$, belongs to ${\text{inv }} (\underrightarrow{\lim}\;C_{b}
(\mathbb{A}_{r}))$, there exists an equivalence class $[(F_{r})]\in{\text{inv
}} (\underrightarrow{\lim}\;C_{b}(\mathbb{A}_{r}))$ such that $\det
(G_{1}^{\ast}G_{2})\cdot\lbrack(F_{r})]=[(1)]$. In particular, this means that
for $r\in(0,1)$ large enough $\det(G_{1}^{\ast}G_{2})$ is bounded and bounded
away from zero. Hence, there exists $F_{r}\in C_{b}(\mathbb{A}_{r})$ such that
$\det(G_{1}^{\ast}(z)G_{2}(z))F_{r}(z)=1$, for $z\in\mathbb{A}_{r}$. That is,
for $\rho$ large enough $\det(G_{1}^{\ast}G_{2})\in{\text{inv }}
(C_{b}(\mathbb{A}_{\rho})).$ Moreover, if $W(\det(G_{1}^{\ast}G_{2}))=0,$
then, arguing as when we verified assumption (A3),
\[
\lim_{r\rightarrow1}w(\det(G_{1}^{\ast}(z)G_{2}(z))|_{\mathbb{T}_{r}})=0.
\]
Due to the local constancy of the winding number, this means that
$w(\det(G_{1}^{\ast}G_{2})|_{\mathbb{T}_{r}})=0$ for all $r$ close enough to
1. That is, $\det(G_{1}^{\ast}G_{2})\in{\text{inv }} C_{b}(\mathbb{A}_{\rho})$
and $w(\det(G_{1}^{\ast}G_{2})|_{\mathbb{T}_{r}})=0,\ r\in(\rho,1)$, for all
$\rho$ close enough to 1. So by Theorem~\ref{tumm}, $d_{v}$ and $\widetilde{d}
_{v}$ coincide in this case.

Next, let us assume that $\widetilde{d}_{v}(P_{1},P_{2})<1$. Then
$\widetilde{d}_{v}^{\rho}(P_{1},P_{2})<1$ for all $\rho$ sufficiently close to
$1$, which means that $\det(G_{1}^{\ast}G_{2})\in{\text{inv }} C_{b}
(\mathbb{A}_{\rho})$ and that $w(\det(G_{1}^{\ast}G_{2})|_{\mathbb{T}_{r}})=0$
for $r\in(\rho,1)$. Therefore, there exists $(F_{r})\in$ $\mathcal{A}$ such
that $\det(G_{1}^{\ast}G_{2})|_{\mathbb{A}_{r}}\cdot F_{r}=1$ pointwise for
$r\in(\rho,1)$. Since $\pi_{r}$ is a $\ast$-homomorphism, this implies that
$\det(G_{1}^{\ast}G_{2})$ is invertible as an element of
$\underrightarrow{\lim}\ C_{b}(\mathbb{A}_{r})$. By definition,
\[
W(\det(G_{1}^{\ast}G_{2}))=\lim_{r\rightarrow1}w(\det(G_{1}^{\ast}
G_{2})|_{\mathbb{T}_{r}}),
\]
and so the assumption that $w(\det(G_{1}^{\ast}G_{2})|_{\mathbb{T}_{r}})=0$,
$r\in(\rho,1),$ for all $\rho$ sufficiently close to $1$ implies that
$W(\det(G_{1}^{\ast}G_{2}))=0$. That is, in view of Theorem~\ref{tumm},
$d_{v}$ and $\widetilde{d}_{v}$ coincides also in this case, which completes
the proof.
\end{proof}

\section{A computational example}

As an illustration of the computability of the proposed $\nu$-metric, we give
an example where we calculate explicitly the $\nu$-metric when there is
uncertainty in the location of the zero of the (nonrational) transfer function.

In \cite{Sas12}, it was shown that
\begin{align*}
d_{\nu}\left(  e^{-sT}\frac{s}{s-a_{1}},e^{-sT}\frac{s}{s-a_{2}}\right)    &
=\frac{|a_{1}-a_{2}|}{\sqrt{2}(a_{1}+a_{2})}\text{ when }|a_{1}-a_{2}|\text{
is small enough, while }\\
d_{\nu}\left(  e^{-sT_{1}}\frac{s}{s-a},e^{-sT_{2}}\frac{s}{s-a}\right)    &
=1\text{ whenever }T_{1}\neq T_{2}.
\end{align*}
Continuing this theme, we will now calculate
\[
d_{\nu}\left(  e^{-sT}\frac{s-a_{1}}{s-b},e^{-sT}\frac{s-a_{2}}{s-b}\right)  ,
\]
hence quantifying the effect of uncertainty in the \emph{zero} location, and
complementing the previous two computations done in \cite{Sas12}, where the
effects of uncertainty in the \emph{pole} location, and uncertainty in the
\emph{delay} were described.

Consider the transfer function $P$ given by
\begin{equation}
P(s):=e^{-sT}\frac{s-a}{s-b},\label{plant}%
\end{equation}
where $T,b>0$, $a\in\mathbb{R}$, and $a\neq b$. Then $P\in\mathbb{F}%
(H^{\infty}(\mathbb{C} _{\scriptscriptstyle >0})),$ where $H^{\infty
}(\mathbb{C} _{\scriptscriptstyle >0})$ denotes the set of bounded and
holomorphic functions defined in the open right half plane
\[
\mathbb{C}_{\scriptscriptstyle >0}:=\{s\in\mathbb{C}:\text{Re}(s)>0\}.
\]
Using the conformal map $\varphi:\mathbb{D}\rightarrow\mathbb{C}
_{\scriptscriptstyle >0}$,
\[
\varphi(z)=\frac{1+z}{1-z},
\]
we can transplant the plant to $\mathbb{D}$. In this manner, we can also talk
about a $\nu$-metric on $\mathbb{S}(H^{\infty}(\mathbb{C}
_{\scriptscriptstyle >0}), p,m)$.

We will calculate the distance between a pair of plants arising from
\eqref{plant}, when there is uncertainty in the parameter $a$, the zero of the
transfer function. A normalized (left and right) coprime factorization of $P$
is given by $P=N/D$, where
\[
N(s)=\frac{(s-b)e^{-sT}}{\sqrt{2}s+\sqrt{a^{2}+b^{2}}},\ \ \ D(s)=\frac
{s-a}{\sqrt{2}s+\sqrt{a^{2}+b^{2}}}.
\]
This factorization was found using the algorithm given in \cite[Example~4.1]%
{PS}. Set $s:=\varphi(z)$ for $z\in\mathbb{D}$, and consider the two plants
\begin{equation}
P_{1}:=e^{-sT}\frac{s-a_{1}}{s-b}\quad\text{and}\quad P_{2}:=e^{-sT}
\frac{s-a_{2}}{s-b},\label{plantor}%
\end{equation}
where $T,b>0$ and $a_{1},a_{2}\in\mathbb{R}\setminus\{b\}$. Define $f$ by
\begin{equation}
f(s) :=G_{1}^{\ast}G_{2}=\overline{N}_{1}N_{2}+\overline{D}_{1} D_{2}
=\frac{(\overline{s}-b)(s-b)e^{-2\operatorname{Re}(s)T}+(\overline{s}
-a_{1})(s-a_{2})}{( \sqrt{2}\overline{s}+\sqrt{a_{1}^{2}+b^{2}}) ( \sqrt
{2}s+\sqrt{a_{2}^{2}+b^{2}}) }.\label{fs}%
\end{equation}
Note that the map $z\mapsto|f(\varphi(z))|$ is bounded on $\mathbb{D}$. We
shall show that the real part of this map is nonnegative and bounded away from
zero for all $s\in\mathbb{C}$ such that $\operatorname{Re}(s)>0,$ provided
that $|a_{1}-a_{2}|$ is small enough. The proof of this fact is analogous to
the proof of \cite[Lemma~4.1]{Sas12}.

\begin{lemma}
Let $T,b>0$ and $a_{1},a_{2}\in\mathbb{R} \setminus\{b\}$. Set $\mathbb{C}%
_{\scriptscriptstyle >0}=\{s\in\mathbb{C}:\operatorname{Re}(s)>0\}$. Let
$f(s)$ be defined as in \eqref{fs}. Then there exist $\delta_{0}$ and $m>0$
such that for all $\delta\in\lbrack0,\delta_{0}),\ s\in\mathbb{C}%
_{\scriptscriptstyle>0}$, there holds that: if $|a_{1}-a_{2}|<\delta$, then
$\operatorname{Re}(f(s))>m>0$.
\end{lemma}

\begin{proof}
Without loss of generality, we may assume that $a_{1}<a_{2}$. Let $a$ and
$\delta$ be such that $a=a_{1}$ and $a+\delta=a_{2}$ respectively, and choose
$\varepsilon>0$ such that $\frac{\varepsilon^{2}}{2}+\frac{3\varepsilon
}{2\sqrt{2}}<\frac{1}{4}$. Note that
\[
\underset{s\in\mathbb{C}_{\scriptscriptstyle >0}}{\lim_{|s|\rightarrow\infty}
}\frac{s-a}{\sqrt{2}s+\sqrt{a^{2}+b^{2}}}=\frac{1}{\sqrt{2}}.
\]
Therefore we can chose $R>0$ such that
\begin{equation}
\left\vert \frac{s-a}{\sqrt{2}s+\sqrt{a^{2}+b^{2}}}-\frac{1}{\sqrt{2}
}\right\vert <\frac{\varepsilon}{2},\label{e1}%
\end{equation}
for all $|s|>R$. We have
\begin{align*}
I & :=  \left|  \frac{s-a}{\sqrt{2}s+\sqrt{a^{2}+b^{2}}}-\frac{s-(a+\delta
)}{\sqrt{2}s+\sqrt{(a+\delta)^{2}+b^{2}}}\right| \\
&  =\left|  \frac{\delta(\sqrt{2}s+\sqrt{a^{2}+b^{2}})+(s-a)(\sqrt
{(a+\delta)^{2}+b^{2}}-\sqrt{a^{2}+b^{2}})}{(\sqrt{2}s+\sqrt{a^{2}+b^{2}
})(\sqrt{2}s+\sqrt{(a+\delta)^{2}+b^{2}})}\right|
\end{align*}
Let us assume that $\delta<|a|$. Then we have
\begin{align*}
I  &  \leq \delta\frac{| \sqrt{2}s+\sqrt{a^{2}+b^{2}}| +3|s-a|}{| \sqrt
{2}s+\sqrt{a^{2}+b^{2}}| \;| \sqrt{2}s+\sqrt{(a+\delta)^{2}+b^{2}}| }\\
&  \leq\delta\Bigg( \frac{1}{|\sqrt{2}s+\sqrt{(a+\delta)^{2}+b^{2} }|
}+\left|  \frac{s-a}{\sqrt{2}s+\sqrt{a^{2}+b^{2}}}\right|  \cdot\frac{3}{|
\sqrt{2}s+\sqrt{(a+\delta)^{2}+b^{2}}| }\Bigg)\\
&  \leq\delta\left(  \frac{1}{\sqrt{2}R}+\left(  \frac{1}{\sqrt{2}}
+\frac{\varepsilon}{2}\right)  \frac{3}{\sqrt{2}R}\right)  ,
\end{align*}
since $\operatorname{Re}(s)>0.$ Therefore, if we choose $\delta_{0}$ so that
for all $0\leq\delta<\delta_{0}$,
\[
\delta<\frac{\sqrt{2}R}{1+\frac{3}{\sqrt{2}}+\frac{3\varepsilon}{2}}\cdot
\frac{\varepsilon}{2},
\]
then for all such $\delta$ and for $|s|>R,$ we have that
\begin{equation}
\left\vert \frac{s-(a+\delta)}{\sqrt{2}s+\sqrt{(a+\delta)^{2}+b^{2}}}-\frac
{1}{\sqrt{2}}\right\vert <\varepsilon.\label{e2}%
\end{equation}
As a consequence of (\ref{e1}) and (\ref{e2}), and using
\[
xy-\frac{1} {2}=\left( x-\frac{1}{\sqrt{2}}\right) \left( y-\frac{1}{\sqrt{2}%
}\right) +\frac{1}{\sqrt{2}}\left( x-\frac{1}{\sqrt{2}}+y-\frac{1}{\sqrt{2}%
}\right) ,
\]
we obtain
\[
\left\vert \frac{s-a}{\sqrt{2}s+\sqrt{a^{2}+b^{2}}}\cdot\frac{s-(a+\delta
)}{\sqrt{2}s+\sqrt{(a+\delta)^{2}+b^{2}}}-\frac{1}{2}\right\vert \leq
\frac{\varepsilon^{2}}{2}+\frac{3\varepsilon}{2\sqrt{2}}.
\]
Therefore,
\[
%\begin{align*}
%&
\frac{1}{2}-\operatorname{Re}\left(  \frac{\overline{s}-a}{\sqrt{2}
\overline{s}+\sqrt{a^{2}+b^{2}}}\cdot\frac{s-(a+\delta)}{\sqrt{2}
s+\sqrt{(a+\delta)^{2}+b^{2}}}\right)
%\\
%&&
%\leq\left\vert \frac{s-a}{\sqrt{2}s+\sqrt{a^{2}+b^{2}}}\cdot\frac
%{s-(a+\delta)}{\sqrt{2}s+\sqrt{(a+\delta)^{2}+b^{2}}}-\frac{1}{2}\right\vert
<\frac{1}{4},
\]
%\end{align*}
and
\[
\operatorname{Re}\left(  \frac{\overline{s}-a}{\sqrt{2}\overline{s}
+\sqrt{a^{2}+b^{2}}}\cdot\frac{s-(a+\delta)}{\sqrt{2}s+\sqrt{(a+\delta
)^{2}+b^{2}}}\right)  >\frac{1}{4}.
\]
Also, for $s\in\mathbb{C}_{\scriptscriptstyle >0},$
\[
\operatorname{Re}\Bigg(  \frac{|s-b|^{2}e^{-2\operatorname{Re}(s)T}}{(
\sqrt{2}\overline{s}+\sqrt{a^{2}+b^{2}}) ( \sqrt{2} s+\sqrt{(a+\delta
)^{2}+b^{2}}) }\Bigg)  >0.
\]
Combined, this means that, for $|s|>R$ and $0\leq\delta<\delta_{0},$
$\operatorname{Re}(f(s))>\frac{1}{4}$.

To treat the case when $|s|<R$, set $K=\{s\in\mathbb{C}%
_{\scriptscriptstyle \geq0}:|s|\leq R\}.$ Define $F:K\times\lbrack
0,1]\rightarrow\mathbb{R}$ by
\[
F(s,\delta)=\operatorname{Re}\left(  \frac{(\overline{s}
-b)(s-b)e^{-2\operatorname{Re}(s)T}+(\overline{s}-a)(s-(a+\delta))}{( \sqrt
{2}\overline{s}+\sqrt{a^{2}+b^{2}}) ( \sqrt{2} s+\sqrt{(a+\delta)^{2}+b^{2}})
}\right)  .
\]
Then
\[
F(s,0)=\operatorname{Re}\left(  \frac{|s-b|e^{-2\operatorname{Re}
(s)T}+|s-a|^{2}}{|\sqrt{2}s+\sqrt{a^{2}+b^{2}}|}\right)  \geq0
\]
and let $2m:=\min_{s\in K}F(s,0).$ Clearly $m\geq0$, and in fact, since $a\neq
b,$ $|s-a|^{2}$ and $|s-b|^{2}$ cannot be zero simultaneously so $m>0$. Now,
since $F$ is continuous on the compact set $K\times\lbrack0,1]$, $F$ is
uniformly continuous there. This means that we may, if necessary, redefine our
choice of $\delta_{0}$ so that if $0\leq\delta<\delta_{0}$, then
$|F(s,\delta)-F(s,0)|<m$ for all $s\in K$. That is, $F(s,\delta
)=\operatorname{Re} (f(s))>m$ for all $\delta\in\lbrack0,\delta_{0}),\ s\in
K$. Combining these observations, we see that $\operatorname{Re}%
(f(s))>\min\{m,1/4\}.$ This completes the proof.
\end{proof}

As a consequence of this result, $G_{1}^{\ast}G_{2}$ is invertible (as an
element of $\underrightarrow{\lim}\;C_{b}(\mathbb{A}_{r})$) and its index is
$W(G_{1}^{\ast}\allowbreak G_{2})=0$. Thus the distance between the two plants
$P_{1},\ P_{2}$ in (\ref{plantor}) is given by
\begin{align*}
\Vert\widetilde{G}_{2}G_{1}\Vert_{\infty} &  =\sup_{s=i\omega,\;\omega
\in\mathbb{R}}\left\vert \frac{(a_{1}-a_{2})(s-b)e^{-sT}}{(\sqrt{2}%
s+\sqrt{a_{1}^{2}+b^{2}})(\sqrt{2}s+\sqrt{a_{2}^{2}+b^{2}})}\right\vert \\
&  =\frac{|a_{1}-a_{2}|}{2}\sup_{\omega\in\mathbb{R}}\frac{\sqrt{\omega
^{2}+b^{2}}}{\sqrt{\omega^{2}+\frac{a_{1}^{2}+b^{2}}{2}}\sqrt{\omega^{2}%
+\frac{a_{2}^{2}+b^{2}}{2}}}.
\end{align*}
We will now determine the supremum in the last expression in the following two
mutually exclusive cases:

\medskip

\noindent$\underline{1}^{\circ}$ Suppose that $(a_{1}^{2}-b^{2})(a_{2}
^{2}-b^{2})\geq4b^{4}$. Then we have
\begin{align*}
\sup_{\omega\in\mathbb{R}}\frac{\sqrt{\omega^{2}+b^{2}}}{\sqrt{\omega
^{2}+\frac{a_{1}^{2}+b^{2}}{2}}\sqrt{\omega^{2}+\frac{a_{2}^{2}+b^{2}}{2}}}
&  =\sup_{\omega\in\mathbb{R}}\frac{1}{\sqrt{(\omega^{2}+b^{2})+\frac{a_{1}
^{2}+a_{2}^{2}-2b^{2}}{2}+\frac{(a_{1}^{2}-b^{2})(a_{2}^{2}-b^{2})}
{4(\omega^{2}+b^{2})}}}\\
&  =\frac{1}{\sqrt{\displaystyle\inf_{\omega\in\mathbb{R}}\left(  \omega
^{2}+b^{2}+\frac{(a_{1}^{2}-b^{2})(a_{2}^{2}-b^{2})}{4(\omega^{2}+b^{2}
)}\right)  +\frac{a_{1}^{2}+a_{2}^{2}-2b^{2}}{2}}}.
\end{align*}
By the arithmetic mean-geometric mean inequality,
\[
\omega^{2}+b^{2}+\frac{(a_{1}^{2}-b^{2})(a_{2}^{2}-b^{2})}{4(\omega^{2}
+b^{2})}\geq\sqrt{(a_{1}^{2}-b^{2})(a_{2}^{2}-b^{2})},
\]
with equality if and only if
\[
\omega^{2}=\sqrt{\frac{(a_{1}^{2}-b^{2})(a_{2}^{2}-b^{2})}{4}}-b^{2}\geq0,
\]
thanks to our assumption that $(a_{1}^{2}-b^{2})(a_{2}^{2}-b^{2})\geq4b^{4}$.
Thus
\[
\sup_{\omega\in\mathbb{R}}\frac{\sqrt{\omega^{2}+b^{2}}}{\sqrt{\omega
^{2}+\frac{a_{1}^{2}+b^{2}}{2}}\sqrt{\omega^{2}+\frac{a_{2}^{2}+b^{2}}{2}}
}=\frac{\sqrt{2}}{\sqrt{a_{1}^{2}-b^{2}}+\sqrt{a_{2}^{2}-b^{2}}}.
\]
Note that in the above, we have used the fact that $a_{1}^{2}>b^{2}$ and
$a_{2}^{2}>b^{2}$ which follows from the condition $(a_{1}^{2}-b^{2}
)(a_{2}^{2}-b^{2})\geq4b^{4}$ ($\geq0$): indeed, if we have (the only other
case) $b^{2}\geq a_{1}^{2}$ and $b^{2}\geq a_{2}^{2}$, then we arrive at the
contradiction that $b^{4}=b^{2}\cdot b^{2}>(b^{2}-a_{1}^{2})(b^{2}-a_{2}
^{2})\geq4b^{4}$.

\medskip

\noindent$\underline{2}^{\circ}$ Now let us consider the other possibility,
namely that $(a_{1}^{2}-b^{2})(a_{2}^{2}-b^{2})< 4b^{4}$. Then for $\omega
\in\mathbb{R}$, we have $4b^{2} (\omega^{2}+b^{2}) \geq4b^{4} > (a_{1}
^{2}-b^{2}) (a_{2}^{2}-b^{2})$ and so
\[
1> \frac{(a_{1}^{2}-b^{2}) (a_{2}^{2}-b^{2})}{4 b^{2} (\omega^{2}+b^{2})}.
\]
From this we obtain upon multiplying both sides by $\omega^{2}$ ($\geq0$)
that
\[
\omega^{2} \geq\frac{(a_{1}^{2}-b^{2}) (a_{2}^{2}-b^{2})\omega^{2}}{4 b^{2}
(\omega^{2}+b^{2})}= \frac{(a_{1}^{2}-b^{2}) (a_{2}^{2}-b^{2})}{4} \left(
\frac{1}{b^{2}}-\frac{1}{\omega^{2}+b^{2}}\right)  .
\]
By rearranging and adding $b^{2}$ on both sides, we have
\[
\omega^{2} +b^{2}+ \frac{(a_{1}^{2}-b^{2}) (a_{2}^{2}-b^{2})}{4 (\omega^{2}
+b^{2})} \geq b^{2}+ \frac{(a_{1}^{2}-b^{2}) (a_{2}^{2}-b^{2})}{4b^{2}},
\]
and so
\begin{align*}
\frac{1}{\sqrt{(\omega^{2}+ b^{2})+ \frac{a_{1}^{2}+a_{2}^{2}-2b^{2}}{2} +
\frac{(a_{1}^{2}-b^{2})(a_{2}^{2}-b^{2})}{4(\omega^{2}+ b^{2}) } }}  &
\leq\frac{1}{\sqrt{b^{2}+ \frac{(a_{1}^{2}-b^{2})(a_{2}^{2}-b^{2})}{4b^{2}}
+\frac{a_{1}^{2}+a_{2}^{2}-2b^{2}}{2} }}\\
&  = \frac{2b}{\sqrt{a_{1}^{2} +b^{2}}\sqrt{a_{2}^{2}+b^{2}}}%
\end{align*}
and there is equality if $\omega=0$. Consequently,
\[
\sup_{\omega\in\mathbb{R}}\frac{\sqrt{\omega^{2}+b^{2}}}{\sqrt{\omega
^{2}+\frac{a_{1}^{2}+b^{2}}{2}}\sqrt{\omega^{2} +\frac{a_{2}^{2}+b^{2}}{2}}} =
\frac{2b}{\sqrt{a_{1}^{2} +b^{2}}\sqrt{a_{2}^{2}+b^{2}}}.
\]

\medskip

Summarizing the two cases, we have
\[
d_{\nu}\left(  e^{-sT} \frac{s-a_{1}}{s-b}, e^{-sT} \frac{s-a_{2}}%
{s-b}\right)  = \left\{
\begin{array}
[c]{ll}%
\displaystyle \frac{|a_{1}-a_{2}|}{\sqrt{2}(\sqrt{a_{1}^{2}-b^{2}}+\sqrt
{a_{2}^{2}-b^{2}})} & \text{if } (a_{1}^{2}-b^{2})(a_{2}^{2}-b^{2})\geq
4b^{4}\\
\displaystyle \frac{b|a_{1}-a_{2}|}{\sqrt{a_{1}^{2} +b^{2}}\sqrt{a_{2}
^{2}+b^{2}}} & \text{if } (a_{1}^{2}-b^{2})(a_{2}^{2}-b^{2})< 4b^{4}%
\end{array}
\right.
\]
when $|a_{1}-a_{2}|$ is small enough.

\section{Our choice of the Banach algebra $S= \displaystyle \lim
_{\longrightarrow}C_{b}(\mathbb{A}_{r})$, and its relation to others}

In this section we give the rationale behind our choice of 
$S$ as the Banach algebra $\underrightarrow{\lim}\;C_{b}(\mathbb{A}_{r})$, 
by first pointing out that some natural guesses for $S$ fail, and this 
is explained in Subsection~\ref{subsec_nat_guesses}. 

Next, in Subsection~\ref{subsec_alternate_description}, we provide another representation  
of  $\underrightarrow{\lim}\;C_{b}(\mathbb{A}_{r})$ using the Stone-\v{C}ech compactification 
of $\mA_0$. 
We also show the relation of our algebra with $L^\infty(\mT)$.  

\subsection{Some natural guesses}
\label{subsec_nat_guesses}

Recall that if $S$ is a unital, commutative Banach algebra, 
then the {\em abstract index}  ${\tt I}_S: \inv S \rightarrow (\inv S)/(\exp S)$ 
is the canonical homomorphism given by ${\tt I}_S(x)=[x]$ for $x\in \inv S$, 
and where $[x]$ denotes the coset of $x$, namely the set $x\exp S $. 

A first guess is to use $L^\infty(\mT)$ as $S$. But then an immediate 
question is what the index function could be. One might guess that the abstract index 
for $L^\infty(\mT)$ works, but as we show below, this is not true. 

\begin{itemize}
 \item[(G1)] $S:=L^\infty(\mT)$ with $\iota:= {\tt I}_{L^\infty(\mT)}$ fails: If we choose $S=L^\infty(\mT)$ and take $\iota$ to be the abstract index function ${\tt I}_{L^\infty(\mT)}$, then the 
``if'' part of (A4) fails, as seen by the following example. 
Consider the element $x:=z\in H^\infty \bigcap \inv L^\infty(\mT)$, and let $g\in L^\infty(\mT)$ be given by 
$g(e^{it})=t$ for $t\in  (-\pi, \pi]$. Then for $t \in (-\pi,\pi]$, $(e^{ig})(e^{it})=e^{i g(e^{it})}= e^{it}$, 
and so $e^{ig}=x$. Thus the abstract index of $x$ is the zero element of the group 
$ (\inv L^\infty(\mT))/(\exp L^\infty(\mT))$, 
but clearly $x=z$ is not invertible as an element of $H^\infty$. 
\end{itemize}

\medskip 

\noindent Another guess might be to  take $S$ to be  $C(\beta \mA_0  \setminus \mA_0)$. Here, $\beta \mA_0$ denotes the 
Stone-\v{C}ech compactification of $\mA_0$, which is the maximal ideal space of the Banach algebra $C_b(\mA_0)$ 
of all complex-valued bounded continuous functions on $\mA_0$. Since $\mA_0$ is 
open in $ \beta \mA_0$, we note that $ \beta \mA_0 \setminus \mA_0$ is compact.  
We show below that also this guess of using 
  $C(\beta \mA_0 \setminus \mA_0)$ fails  if we use the abstract index. 

\begin{itemize}
 \item[(G2)] $S:= C(\beta \mA_0 \setminus \mA_0)$ with 
$\iota:={\tt I}_{ C(\beta \mA_0 \setminus \mA_0)}$ fails: We take $S$ to be the $C^*$-algebra of continuous functions $C(Y)$ on  
$Y:=\beta \mA_0 \setminus \mA_0$, again with the abstract 
index defined on $C(Y)$ taken as a candidate 
choice for the index $\iota$. Then one can see that  the ``only if'' part in (A4) fails by means of 
the following example. 

Let $c$ be any conformal 
mapping from the unit disk $\mD:=\{z\in \mC: |z|<1\}$ onto the 
strip $\{z\in \mC: -1<\textrm{Re}(z)<0\}$. Set $x:=e^c$. Then $x$ is an element 
of $C_b(\mA_0)$ because $|e^c|=e^{\textrm{Re}(c(z))}<e^0=1$ for $z\in 
\mA_0$. Also, $x$ is invertible as an element of $C_b(\mA_0)$, because its inverse is $e^{-c}$, which belongs 
to $C_b(\mA_0)$: $|e^{-c}|=e^{-\textrm{Re}(c(z))}<e^1=e$ for $z\in 
\mA_0$. But we will now show that the abstract index of $x$ is not 
the zero element of $(\inv C(Y)) / (\exp C(Y))$, by showing that it does not admit a continuous logarithm $G$ on 
$Y$. Suppose, on the contrary, that it does admit 
a continuous logarithm $G$ on $Y$: $\widehat{x}= e^{G}$. Then it also admits a continuous logarithm on a 
neighbourhood of $Y$ in $\beta \mA_0$.  This implies that 
$e^{c(z)}=e^{G(z)}$ for all $z\in \mA_r$, for a large enough $r$, and so $c(z)-G(z)=2\pi i k(z)$ for some integer-valued 
function $k$ on $\mA_r$. But as $c,G$ are continuous, the above shows that $k$ should be continuous too, 
and being integer-valued, it must reduce to a constant. But this is impossible, as $G$ is bounded while 
the imaginary part of $c$ can be made as large as we please on $\mA_r$.
\end{itemize}

\subsection{The isometric isomorphism $ \displaystyle \lim
_{\longrightarrow}C_{b}(\mathbb{A}_{r}) \simeq C(\beta \mA_0 \setminus \mA_0)$} 
\label{subsec_alternate_description}

$\;$

\noindent We now demostrate that in our above guess (G2), what was wrong was 
the choice of the index function. Indeed, we  show below that 
 $
\underrightarrow{\lim}\;C_{b}(\mathbb{A}_{r}) \simeq C(\beta \mA_0 \setminus \mA_0)$, 
that is, the two are isometrically isomorphic. (Thus rather than using the abstract index for $C(\beta \mA_0 \setminus \mA_0)$, which fails as shown in (G2) above, 
one should use the index $W$ defined via winding numbers, which works.)  This gives another description of our Banach algebra, 
$\underrightarrow{\lim}\;C_{b}(\mathbb{A}_{r}) $, albeit it is not as concrete as our earlier description, since it relies on the Stone-\v{C}ech compactification of $\mA_0$. 
For the purposes of doing computation, for instance the calculation of $W$, it is more explicit to work with 
$\underrightarrow{\lim}\;C_{b}(\mathbb{A}_{r})$, where one just has concrete sequences of functions on shrinking annuli.

\begin{theorem}
 $\underrightarrow{\lim}\;C_{b}(\mathbb{A}_{r}) $ is isometrically isomorphic to 
 $ C(\beta\mA_0 \setminus \mA_0)$. 
\end{theorem}
\begin{proof}
 Consider the restriction map $r: C(\beta\mA_0) \rightarrow C(\beta\mA_0 \setminus \mA_0)$. 
This map is clearly onto. 

Let $f\in C(\beta\mA_0 \setminus \mA_0)$. Suppose that $F,G \in 
 C(\beta\mA_0)$  
 are such that $r(F)=f=r(G)$. Then since $F$ and $G$ coincide on the compact set 
$\beta\mA_0 \setminus \mA_0$, 
it follows that given any $\epsilon>0$, there is a $r<1$, large enough, 
such that $|F-G|<\epsilon$ on $\mA_r$. So $\pi_0(F)=\pi_0(G)$. Hence from $f$ 
we can obtain a well-defined element $\widetilde{f} \in \underrightarrow{\lim}\;C_{b}(\mathbb{A}_{r}) $, by setting 
 $ \widetilde{f}:= \pi_0(F)$, where $r(F)=f$. 
Also, the above shows that 
$$
\|\widetilde{f}\|_{\underrightarrow{\lim}\;C_{b}(\mathbb{A}_{r}) }
= \|f\|_{C(\beta\mA_0 \setminus \mA_0)}.
$$ 
Hence we have the isometric embedding 
$$
f \mapsto \widetilde{f} : C(\beta\mA_0 \setminus \mA_0) \hookrightarrow 
\underrightarrow{\lim}\;C_{b}(\mathbb{A}_{r}).
$$
Next, we show that the above mapping is surjective. To this end, take any 
$[(f_r)] \in \underrightarrow{\lim}\;C_{b}(\mathbb{A}_{r})$. We can take any $f_r \in C_b(\mA_r)$ (with $r$ 
large enough so that 
each $f_R$ for $R>r$ is a restriction of $f_r$) and associate with such an $f_r$ an element 
$f\in C_b(\mA_0)$ by setting 
$$
f(z)= \left\{ \begin{array}{ll} f(z) & \textrm{for } \frac{1+r}{2} \leq |z|<1 , \\
f_r\left(\left( r+|z| \cdot \frac{1-r}{1+r} \right) \frac{z}{|z|} \right) &  
\textrm{for } 0<|z|<\frac{1+r}{2},  \end{array}\right.
$$
 and then we can restrict this $f$ to $C(\beta \mA_0 \setminus \mA_0)$.  
This function in 
$C(\beta \mA_0 \setminus \mA_0)$, which we denote again by $f$, maps to a well-defined element $\widetilde{f}\in 
 \underrightarrow{\lim}\;C_{b}(\mathbb{A}_{r})$, and it follows from our construction that $\widetilde{f}= [(f_r)]$. 
\end{proof}

% \subsection{Relation to other algebras}
% \label{subsec_relation_to_others}
% 
% The Douglas algebra $H^\infty+C(\mT)$ is not closed under the involution it inherits from $L^\infty(\mT)$, 
% and so it is not suitable as a choice of $S$ in our set-up. However, we show that $H^\infty+C(\mT)$ can be considered 
% to be a subalgebra of $\underrightarrow{\lim}\;C_{b}(\mathbb{A}_{r})$ as follows. Firstly, 
% $C(\mT)$ is isometrically 
% embedded in $\underrightarrow{\lim}\;C_{b}(\mathbb{A}_{r})$, and this can be seen as follows: for $f\in C(\mT)$, 
% define $f_r\in C_b(\mA_r)$ by 
% $$
% f_r(z)=f\left(\frac{z}{|z|}\right), \quad z\in \mA_r, \;0<r<1. 
% $$
% So it is clear that there is an injective Banach algebra homomorphism from 
% $H^\infty +C(\mT)$ to $\underrightarrow{\lim}\;C_{b}(\mathbb{A}_{r})$.
 
In order to view $\underrightarrow{\lim}\; C_{b}(\mathbb{A}_{r})$ as a subalgebra of $L^\infty(\mT)$, we 
will recall first the definition of the generalized argument principle for a closed subset $Y$ of 
the maximal ideal space ${\mathfrak{M}}(S)$ of a Banach algebra $S$. This notion was introduced in  
 \cite[Definition~2.1]{Moh}. 

\begin{definition}
  As before, let ${\mathfrak{M}}(S)$ denote the maximal ideal space of a unital commutative
  Banach algebra $S$. A closed subset $Y \subset {\mathfrak{M}}(S)$ is said to
  satisfy the {\em generalized argument principle} for $S$ if whenever
  $a \in S$ and $\log \widehat{a} $ is defined continuously on $Y$,
  then $a$ is invertible in $S$. (Here $\widehat{a}$ denotes the
  Gelfand transform of $a$, $Y$ is equipped with the topology it
  inherits from ${\mathfrak{M}}(S)$ and ${\mathfrak{M}}(S)$ has the usual Gelfand topology.)
\end{definition}

It was shown in \cite[Theorem~2.2]{Moh} that any $Y$ satisfying the
generalized argument principle is a boundary for $S$ and so it
contains the \v{S}ilov boundary of $S$.

We have  
$$
Y:={\mathfrak{M}}\left(\underrightarrow{\lim}\; C_{b}(\mathbb{A}_{r})\right)\subset 
{\mathfrak{M}}(H^\infty).
$$
Now suppose that $f\in H^\infty \bigcap \inv\left( \underrightarrow{\lim}\; C_{b}(\mathbb{A}_{r})\right)$, 
and that $f=e^G$ for some $G\in C(Y)$. Then ${\tt I}_{C(Y)}(f)$ is the zero element of the group 
$\inv \left(\displaystyle \lim_{\longrightarrow} C_b(\mA_r)\right)/ 
\exp \left(\displaystyle \lim_{\longrightarrow} C_b(\mA_r)\right)$. But we know from \cite[p.263]{Mur}, 
that the topological index $W$ is of the form $W 
=\varphi \circ {\tt I}_{C(Y)}$ 
for some group homomorphism $\varphi$ from $\inv \left(\displaystyle \lim_{\longrightarrow} C_b(\mA_r)\right)/ 
\exp \left(\displaystyle \lim_{\longrightarrow} C_b(\mA_r)\right)$ to $\mZ$. Thus it follows that $W(f)=0$, and so 
$f$ is invertible as an element of $H^\infty$. So the set $Y$ satisfies the generalized argument principle. 
Hence $Y$ contains  the \v{S}ilov boundary of $H^\infty$, which in turn can be identified with the maximal 
ideal space of $L^\infty(\mT)$. 

Hence if $f\in \underrightarrow{\lim}\; C_{b}(\mathbb{A}_{r})$, then $\widehat{f} \in C(Y)$, and 
so $\widehat{f}|_{{\mathfrak{M}}(L^\infty(\mT))}$ determines an element of $L^\infty(\mT)$.  Thus we have 
 injective Banach algebra homomorphisms 
from $\underrightarrow{\lim}\; C_{b}(\mathbb{A}_{r})$ to $L^\infty(\mT)$: 
 $
H^\infty  \subset \underrightarrow{\lim}\; C_{b}(\mathbb{A}_{r}) \subset L^\infty(\mT).
$

We remark that by taking the harmonic extension of $f\in L^\infty(\mT)$ \cite[Lemma~6.43]{Dou}, we obtain 
a continous function $f^{h}$ on the disk, which gives an element $\pi_0(f^h)$ of 
$\underrightarrow{\lim}\;C_{b}(\mathbb{A}_{r})$. However, although the map $f\mapsto f^h$ 
is an injective linear contraction \cite[Lemma~6.44]{Dou}, it is not a Banach algebra 
homomorphism, because the map is not  multiplicative.

\end{document}